\documentclass[12pt, reqno]{amsart}

\usepackage{amsmath,times,epsfig,amssymb,amsbsy,amscd,amsfonts,amstext,color,bm}
\usepackage[arrow,matrix]{xy}
\usepackage{tikz}
\usepackage{amsthm}
\usepackage{enumitem}
\usepackage[numbers]{natbib}
\numberwithin{equation}{section}

\theoremstyle{plain}
\newtheorem{theorem}{Theorem}[section]

\newtheorem{lemma}[theorem]{Lemma}

\newtheorem{proposition}[theorem]{Proposition}

\theoremstyle{definition}
\newtheorem{definition}[theorem]{Definition}

\theoremstyle{remark}
\newtheorem{remark}[theorem]{Remark}

\newcommand{\A}{\mathcal{A}}

\newcommand{\D}{\mathcal{D}}
\newcommand{\Z}{\mathbb{Z}}

\newcommand{\C}{\mathbb{C}}

\newcommand{\res}{\mathbf{res}}

\newcommand{\Cat}{\operatorname{Cat}}
\newcommand{\Shi}{\operatorname{Shi}}
\newcommand{\Der}{\operatorname{Der}}

\newcommand{\cShi}{\mathbf{c}\Shi}

\newcommand{\cCat}{\mathbf{c}\Cat}

\begin{document}
	
	\title{On the extended Shi arrangement of type $B_{\ell}$}
	
	\begin{abstract}
		Motivated by Kawanoue's construction of a basis for
		the derivation module of the cone of 
		the extended Catalan arrangement of type $B_\ell$, 
		we construct an explicit basis for the
		derivation module of $\cShi(B_\ell,m)$. 
		Our construction generalizes the
		previously known case $m=1$ to $m\geq1$. 
		It also provides a partial answer to a question of Abe and Terao 
		concerning explicit constructions of simple-root bases.	
	\end{abstract}
	
	\author{Zixuan Wang}
	\address{Zixuan Wang, 
		Nanjing Forestry University
	}
	\email{zixuan.wang@njfu.edu.cn}
	


	\begin{abstract}
		Motivated by Kawanoue's construction of a basis for
		the derivation module of the cone of 
		the extended Catalan arrangement of type $B_\ell$, 
		we construct an explicit basis for the
		derivation module of $\cShi(B_\ell,m)$. 
		Our construction generalizes the
		previously known case $m=1$ to $m\geq1$. 
		It also provides a partial answer to a question of Abe and Terao 
		concerning explicit constructions of simple-root bases.		
	\end{abstract}	
	
	\maketitle

	\section{Introduction}
	Let $V=\oplus_{i=1}^\ell  \C e_i$ be a vector space 
	with dual $V^* =\oplus_{i=1}^\ell  \C x_i $,
	and let $S=\C[x_1,\dots,x_\ell]$ be the polynomial ring.  
	The hyperplane $H$ is an affine subspace of dimension $(\ell-1)$,
	which can be written as the kernel of a polynomial $\alpha_H$ of degree $1$.
	The hyperplane arrangement $\A$ is a finite set of hyperplanes.
	Call $\A$ a central arrangement if all hyperplanes pass through the origin.
	Write $
	\Der (S)=\bigoplus_{i=1}^{\ell}S\partial_{x_i}
	$.
	The derivation module $\D(\A)$ of a central arrangement $\A$ is defined by
	\begin{equation*}
		\D(\A)=\{ \theta \in \mathrm{Der}(S) \mid \theta(\alpha_H) \in \alpha_H S, \text{for \ any \ } H \in \A \}.
	\end{equation*}
	The central arrangement $\A$ is said to be free if $\D(\A)$ is free.
	A derivation is homogeneous of degree $q$ if all its nonzero coefficient
	polynomials are homogeneous of degree $q$.
	The degrees of derivations in a basis of $\D(\A)$ form the exponents of $\A$. 
	For more details, see \cite{OT}.
	
	Let $\Phi$ be a crystallographic root system with a fixed positive system
	$\Phi^+$. 
	After adjoining a coordinate $z$, define the cone of the extended Shi
	arrangement for $m\geq1$ to be the central arrangement
	\[
	\cShi(\Phi,m)=\{z=0\}\cup
	\{\alpha-kz=0 \mid \alpha\in\Phi^+,\ 1-m\leq k\leq m\}.
	\]
	Its deconing in the affine subspace $\{z=1\}$ is the extended Shi arrangement.  
	The case $m=1$ is the Shi arrangement,
	which originated in Shi's study of Kazhdan--Lusztig cells in affine Weyl groups \cite{S}.
	
	Set
	$
	\widetilde S=S[z].
	$
	Write
	$
	\Der\widetilde S=
	\bigoplus_{i=1}^{\ell}\widetilde S\partial_{x_i}
	\oplus\widetilde S\partial_z.
	$
	Choose
	\begin{equation}\label{eq:positive-roots}
		\Phi^+(B_\ell)=
		\{x_i \mid 1\leq i\leq\ell\}
		\cup\{x_i \pm x_j \mid 1\leq i<j\leq\ell\}.
	\end{equation}
	The Coxeter number of the Weyl group of type $B_\ell$ is $h=2\ell$.
	Let $d=mh=2m\ell$.
	Set
	\begin{equation}\label{eq:CR}
		C_m(t,z)=\prod_{k=-m}^{m}(t-kz),
		\qquad
		R_m(t,z)=\prod_{k=1-m}^{m}(t-kz).
	\end{equation}
	The cone of the extended Shi arrangement $\cShi(B_\ell,m)$ is the central
	arrangement in $\C^{\ell+1}$ with defining polynomial
	$
	Q_{\cShi}(x,z)=z\prod_{\alpha\in\Phi^+(B_\ell)}R_m(\alpha(x),z).
	$
	Thus its defining hyperplanes are $\{z=0\}$ and
	$\{\alpha-kz=0\}$, where $\alpha\in\Phi^+(B_\ell)$ and
	$1-m\leq k\leq m$.  
	For a graded module $M$, write
	$M_q$ for its homogeneous degree-$q$ component.
	Let
	$
	\D_0(\cShi(\Phi,m))=
	\{\theta \in \D(\cShi(\Phi,m)) \mid \theta (z)=0\}.
	$
	When $\ell$, $m$, and $\Phi$ are fixed, we abbreviate this module to
	$\D_0(\cShi)$.
	The Euler derivation is
	$
	\theta_E=z\partial_z+\sum_{i=1}^{\ell}x_i\partial_{x_i}.
	$
	
	The freeness of the cones of the extended Shi and extended Catalan arrangements 
	was conjectured by Edelman and Reiner \cite{ER}.  
	The cases of type $A_{\ell-1}$ were obtained 
	in work of Edelman--Reiner and Athanasiadis \cite{Ath,ER}, 
	and Yoshinaga \cite{Y04} proved the conjecture for every crystallographic root system. 
	
	Further developments have been made concerning the construction of bases.
	For the case $m=1$,  
	Suyama and Terao \cite{ST} constructed a basis for the Shi arrangement of type $A_{\ell-1}$, 
	and Suyama \cite{S15} treated it 
	in types $B_\ell$ and $C_\ell$ by Bernoulli-like polynomials.  
	Abe and Terao \cite{AT} introduced simple-root bases 
	for the extended Shi arrangements 
	and characterized them through properties of the Weyl group. 
	They posed the problem of finding the explicit construction of
	simple-root bases for every root system.
	For type $A_{\ell-1}$, 
	Suyama and Yoshinaga \cite{SY} obtained explicit bases from discrete integrals.  
	More recently, 
	Feigin, Wang, and Yoshinaga \cite{FWY}
	gave the integral expressions for derivations of free multiarrangements
	related to complex reflection groups,
	and gave a conjecture concerning an explicit basis of $\cCat(B_2, m)$.  
	Kawanoue \cite{K}
	proved it and extended the results to
	$\cCat(B_\ell,m)$.  
	
	Kawanoue \cite{K} defined the following polynomials and vector fields to construct the basis for the extended Catalan arrangement of type $B_\ell$.
	For $r\in\Z_{\geq0}$, 
	set
	\begin{equation}
		\label{P-poly}
		P_z(a,r)=\prod_{k=-r}^{r}(a-kz).
	\end{equation}
	For $n\in \C$ and $q\in\Z_{\geq0}$, 
	use
	$\binom{n}{q}=n(n-1)\cdots(n-q+1)/q!$.
	For $0\leq c_1,\dots,c_{\ell-1}\leq m$, 
	set
	$C_j=c_1+\cdots+c_j$ and $C_0=0$, 
	and define
	\begin{multline}\label{eq:k-h}
		\widetilde h_{\ell-1,m,0}(x;y_1,\dots,y_{\ell-1};z) =\\
		\sum_{0\leq c_1,\dots,c_{\ell-1}\leq m}
		\frac{(-1)^{C_{\ell-1}}P_z(x,m+C_{\ell-1})}
		{\binom{m+C_{\ell-1}+\frac12}{m+1}}
		\prod_{j=1}^{\ell-1}\binom{m}{c_j}
		\frac{P_z(y_j,m+C_{j-1})}{P_z(y_j,C_j)}.
	\end{multline}
	The quotients are polynomials because
	$m+C_{j-1}-C_j=m-c_j\geq0$.  Every summand has degree $d+1$.
	Set
	\begin{equation}
		\label{eq:k-basis}
		\begin{aligned}
			\widetilde F_{\ell,m,0,i}(x,z)
			&=\widetilde h_{\ell-1,m,0}
			(x_i;x_1,\dots,\widehat{x_i},\dots,x_\ell;z),\\
			\widetilde\kappa_{\ell,m}
			&=\sum_{i=1}^{\ell}\widetilde F_{\ell,m,0,i}(x,z)\partial_{x_i}.
		\end{aligned}
	\end{equation}
	Here a hat denotes omission.  
	This is the degree-$(d+1)$ 
	generator 
	in Kawanoue's homogeneous basis 
	for the extended Catalan arrangement of type $B_\ell$.
	
	For $v\in\C^\ell$ and a polynomial $f(x,z)$,
	define
	\begin{equation}\label{eq:dif}
		\rho_v^z f(x,z)=\frac{f(x,z)-f(x-zv,z)}{z}.
	\end{equation}
	Note that $f(x ,z) - f(x-z v, z) \equiv f(x, 0) - f(x, 0)=0$ $(\mathrm{mod}\  z)$.
	So the quotient is a polynomial.
	If $f$
	is homogeneous of degree $q$, 
	then $\rho_v^z f$ is homogeneous of degree
	$q-1$.
	We can apply $\rho_v^z$ coefficientwise to derivations that annihilate $z$.
	
	For $1\leq r<\ell$, 
	set $\alpha_r=x_r-x_{r+1}$, 
	let
	$\sigma_r=(r\ r+1)$, 
	and let $s_r$ interchange $x_r$ and $x_{r+1}$ 
	while fixing $z$.  
	Its action on the derivation module is
	\begin{equation}
		\label{eq:reflection-action}
		s_r\left(\sum_i^{\ell} f_i\partial_{x_i}+g\partial_z\right)
		=\sum_{i=1}^{\ell}s_r(f_i)\partial_{x_{\sigma_r(i)}}+s_r(g)\partial_z.
	\end{equation}
	Set
	$
	\lambda_{\ell,m}=\frac{(-1)^{\ell m+1}}{4m(m+1)},
	\Phi_0^{(m)}=0,
	\Phi_1^{(m)}=\lambda_{\ell,m}\rho_{e_1}^z
	\widetilde\kappa_{\ell,m}.
	$
	For $1\leq r<\ell$, 
	define recursively in the fraction field of
	$\widetilde S$ by
	\begin{equation}\label{eq:Phi-recursion}
		\Phi_{r+1}^{(m)}=
		\frac{(\alpha_r+mz)\Phi_r^{(m)}
			-(mz-\alpha_r)s_r\Phi_r^{(m)}}{\alpha_r}
		-\Phi_{r-1}^{(m)},
	\end{equation}
	and set
	\begin{equation}\label{eq:Theta-explicit}
		\Theta_i^{(m)}=\Phi_i^{(m)}-\Phi_{i-1}^{(m)}
		\qquad(1\leq i\leq\ell).
	\end{equation}
	
	With the above definitions, the main result of this paper is as follows.
	\begin{theorem}
		\label{thm:main}
		For $\ell\geq2$ and $m\geq1$, the vector fields 
		$
		\theta_E,\Theta_1^{(m)},\dots,\Theta_\ell^{(m)}
		$
		form a homogeneous $\widetilde S$-basis for
		$\D(\cShi(B_\ell,m))$.  
	\end{theorem}
	
	In this paper, we fix integers $m, \ell$ and the positive system $\Phi^+(B_\ell)$.
	The paper is organized as follows.
	In Section \ref{subsec:hyperplane},
	we introduce the Ziegler restriction and the related useful theorems.
	In Section \ref{subsec:Weyl-group},
	we introduce the Weyl group and its action. 
	We also show some properties of the simple-root basis from Abe and Terao.
	Then we give the proof of the main theorem \ref{thm:main} in Section \ref{sec:proof}.
	In the final section, we give an example for the case $B_2$ and $m=1$.
	
	\section{Preliminaries}
	
	\subsection{Hyperplane arrangement}\label{subsec:hyperplane}
	Let $\A$ be a nonempty arrangement and $H\in \A$.
	Define the restriction arrangement 
	by $\A^H=\{ H' \cap H \mid H'\in \A \backslash \{H\}, H' \cap H \neq \emptyset \}$.
	For a central arrangement $\A$ 
	and a map 
	$m:\A\longrightarrow\Z_{\geq 0}$,
	the pair $(\A, m)$ is called a multiarrangement. 
	The defining polynomial $Q(\A, m)$ of $(\A, m)$ is
	$
	Q(\A, m)=\prod_{H\in\A}\alpha_H^{m(H)}
	$. 	
	For a multiarrangement $(\A, m)$, 
	define the derivation module $\D(\A, m)$ by 
	\begin{equation}
		\D(\A, m)=\{\theta\in\Der(S)\mid \theta(\alpha_H)\in \alpha_H^{m(H)} S, 
		\mbox{ for any }H\in\A\}. 
	\end{equation}
	The multiarrangement $(\A, m)$ is free, if $\D(\A, m)$ is a free $S$-module.
	When $m \equiv 1$, $\D(\A, 1) =\D(\A)$.   
	
	\begin{proposition}
		\cite{OT,Z}
		\label{prop:szc}
		(Saito--Ziegler criterion) 
		Let $(\A, m)$ be a multiarrangement. 
		For homogeneous vector fields 
		$\theta_1, \dots, \theta_\ell\in \D(\A, m)$,
		the multiarrangement $(\A, m)$ is free with a basis 
		$\{\theta_1, \dots, \theta_\ell\}$
		if and only if $\theta_1, \dots, \theta_\ell$ are linearly independent over $S$,
		and $\sum_{i=1}^\ell\deg\theta_i=\sum_{H\in\A} m(H)$.
	\end{proposition}
	
	\begin{definition} 
		For a central arrangement $\A$ and a hyperplane $H\in \A$, let $(\A^H, m^H)$ be the Ziegler restriction 
		of $\A$ onto $H$ where $m^H$ is defined by
		$
		m^H(H'\cap H)= | \{ H''\in \A\backslash \{H\} \mid H''\cap H = H'\cap H \} |. 
		$
	\end{definition}
	
	The Ziegler restriction of $\cShi(B_\ell,m)$ onto $\{z=0\}$ is the Coxeter
	multiarrangement $(B_\ell,2m)$. 
	It has defining polynomial
	\begin{equation}\label{eq:B-multi}
		Q(B_\ell,2m)=
		\prod_{i=1}^{\ell}x_i^{2m}
		\prod_{1\leq i<j\leq\ell}(x_i^2-x_j^2)^{2m}.
	\end{equation}
	
	For a hyperplane $H$ and a vector field $\theta \in \D(\A)$ satisfy $\theta (\alpha_H)=0$,
	we have the restriction $\theta |_H \in \D(\A^H, m^H)$ by \cite{Z}.
	In particular,
	for any $\eta=\sum_i f_i(x,z)\partial_{x_i}\in\Der\widetilde S$ satisfying
	$\eta(z)=0$, define coefficientwise restriction by
	\begin{equation}
		\label{eq:res}
		\res(\eta)=\sum_{i=1}^{\ell}f_i(x,0)\partial_{x_i}.
	\end{equation}
	It can be applied to polynomials:
	$\res (f) = f(x, 0)$, for $f \in \widetilde{S}$.
	Yoshinaga \cite{Y04} showed the homogeneous part
	\begin{equation}
		\label{eq:res-d}
		\res_d:\D_0(\cShi(B_\ell, m))_d \longrightarrow \D(B_\ell, 2m)_d
	\end{equation}
	is a linear isomorphism.
	For $1\leq i\leq\ell$, define the vector field
	\begin{equation}\label{eq:theta-integral}
		\theta_i^{(m)}
		=x_i\sum_{j=1}^{\ell}
		\left(
		\int_0^{x_j}
		\frac{\prod_{k=1}^{\ell}(t^2-x_k^2)^m}{t^2-x_i^2}
		d t
		\right)\partial_{x_j}.
	\end{equation}
	The quotient in the integrand is polynomial because $m\geq1$.
	
	\begin{theorem}
		\cite{FWY}
		\label{thm:FWY}
		The derivations $\theta_1^{(m)},\dots,\theta_\ell^{(m)}$ defined above form a basis of
		$\D(B_\ell,2m)$. Each has degree $d=2m\ell$. 
	\end{theorem}
	
	For vector field $\theta$ and $\delta$ with $\delta = \sum_{i=1}^\ell f_i \partial_{x_i}$, 
	define an affine connection $\bigtriangledown: \Der(S) \times \Der(S) \longrightarrow \Der(S)$ by
	\begin{equation}
		\label{eq:connection}
		\bigtriangledown_\theta \delta =\sum_{i=1}^\ell \theta(f_i) \partial_{x_i}.
	\end{equation}
	Introduce 
	\begin{equation}
		\label{eq:eta0}
		\eta_0^{(m)}
		=\sum_{i=1}^{\ell}
		\left(\int_0^{x_i}\prod_{k=1}^{\ell}(t^2-x_k^2)^m  dt\right)
		\partial_{x_i}.
	\end{equation}
	
	\begin{lemma}
		\label{lem:eta0}
		For $1\leq i\leq\ell$,
		$
		\bigtriangledown_{\partial_{x_i}}\eta_0^{(m)}=-2m\theta_i^{(m)}.
		$
	\end{lemma}
	\begin{proof}
		For each $\int_0^{x_j}\prod_k(t^2-x_k^2)^m d t$, we have
		\[
		\partial_{x_i} \int_0^{x_j}\prod_k(t^2-x_k^2)^m d t =-2mx_i\int_0^{x_j}
		\frac{\prod_k(t^2-x_k^2)^m}{t^2-x_i^2} d t.
		\]
		Summing over $j$ proves the assertion.
	\end{proof}
	
	\begin{theorem}
		\cite{Y14,Z}
		\label{thm:restriction}
		For an arrangement $\A$ and a hyperplane $H$ in $\C^{\ell+1}$, 
		assume that homogeneous vector fields $\delta_1, \delta_2, \dots, \delta_{\ell} \in \D(\A)$ satisfy $\delta_i (\alpha_H)=0$.
		If the restrictions $\delta_1 |_H, \delta_2 |_H, \dots, \delta_{\ell} |_H$ form a basis for the 
		Ziegler restriction $(\A^H, m^H)$, 
		$\delta_1, \delta_2, \dots, \delta_{\ell}$ 
		together with the Euler derivation $\theta_E$ form a basis for $\A$ 
		with exponents $(1, \deg \delta_1, \dots, \deg \delta_{\ell})$.
	\end{theorem}
	
	\subsection{The Weyl group and its action}\label{subsec:Weyl-group}
	
	The Weyl group of type $B_\ell$ is the signed permutation group
	$
	W=W(B_\ell) \cong \{\pm1\}^\ell\rtimes\mathfrak{S}_\ell.
	$
	Thus $w\in W$ is uniquely described by a permutation
	$\sigma\in\mathfrak S_\ell$ and signs
	$\varepsilon_1,\dots,\varepsilon_\ell\in\{\pm1\}$ such that
	$
	w(e_i)=\varepsilon_i e_{\sigma(i)}.
	$
	See \cite{H} for details.
	\begin{definition}
		[Definition 6.2, \cite{OT}]
		\label{def:action-on-module}
		Let $w\in W$, $f\in S$, $v\in V$, $\theta \in\Der(S)$. 
		Define the $W$-module structure
		\begin{itemize}
			\item 
			in $S$ by $(wf)(v)=f(w^{-1}v),$  
			\item 
			in $\Der(S)$ by $(w\theta)(f)=w\bigl(\theta(w^{-1}f)\bigr)$.  
		\end{itemize} 
	\end{definition}
	We extend the action to $\widetilde{S}$ by declaring that $w(z)=z$.
	In particular,
	$
	w(\partial_{x_i})=\varepsilon_i\partial_{x_{\sigma(i)}},
	w(\partial_z)=\partial_z,
	$
	and hence
	\begin{equation}
		\label{eq:action-on-field}
		w\left(\sum_{i=1}^{\ell}f_i\partial_{x_i}+g\partial_z\right)
		=\sum_{i=1}^{\ell}
		\varepsilon_i w(f_i)\partial_{x_{\sigma(i)}}+w(g)\partial_z.
	\end{equation}
	For $v\in V$,
	we have
	$w(\partial_v)=\partial_{wv}$.
	Choose the simple roots
	$\alpha_r=x_r-x_{r+1}\ (1\leq r<\ell),\alpha_\ell=x_\ell.$
	For $1\leq r<\ell$, 
	the corresponding simple reflection $s_r$ swaps
	$e_r$ with $e_{r+1}$, swaps $x_r$ with $x_{r+1}$, 
	and fixes all other $e_i,x_i$ and $z$.  
	The final simple reflection $s_\ell$ changes the signs of
	$e_\ell,x_\ell$, 
	and $\partial_{x_\ell}$ 
	and fixes the remaining coordinates
	and $z$.
	
	Equip $V^*$ with the standard $W$-invariant inner product for which
	$x_1,\dots,x_\ell$ are orthonormal.  The action of a simple reflection
	on $V^*$ is
	\begin{equation}
		\label{eq:simple-reflection}
		s_r(\beta)
		=\beta-
		\frac{2(\alpha_r,\beta)}{(\alpha_r,\alpha_r)}\alpha_r
		\qquad(\beta\in V^*).
	\end{equation}
	
	If $\sigma_r=(r\ r+1)$, then
	\begin{equation}
		\label{eq:s-r-action}
		s_r\left(\sum_i f_i\partial_{x_i}+g\partial_z\right)
		=\sum_i s_r(f_i)\partial_{x_{\sigma_r(i)}}+s_r(g)\partial_z.
	\end{equation}
	
	\begin{definition}
		\label{eq:Xi}
		Define the map
		$\Xi_m:V\longrightarrow \D(B_\ell,2m)_d$
		by
		$\Xi_m(e_i)=\theta_i^{(m)}.$
	\end{definition}
	
	\begin{lemma}
		\label{lem:Xi-equiv}
		The map $\Xi_m$ is a $W$-isomorphism.
	\end{lemma}
	
	\begin{proof}
		Let $w\in W$.
		Since $\prod_{k=1}^{\ell}(t^2-x_k^2)^m$ is $W$-invariant
		under signed permutations of the $x$-variables,
		we have
		\[
		w(\int_0^{x_j}\frac{\prod_{k=1}^{\ell}(t^2-x_k^2)^m}{t^2-x_i^2}dt)=
		\int_0^{\varepsilon_j x_{\sigma(j)}} \frac{\prod_{k=1}^{\ell}(t^2-x_k^2)^m}{t^2-x_{\sigma(i)}^2}dt.
		\]
		Since $\frac{\prod_{k=1}^{\ell}(t^2-x_k^2)^m}{t^2-x_i^2}$ is even in $t$,
		we have
		\begin{align*}
			w(\theta_i^{(m)})
			&=\varepsilon_i x_{\sigma(i)}
			\sum_{j=1}^{\ell}
			\left(
			\int_0^{\varepsilon_j x_{\sigma(j)}} 
			\frac{\prod_{k=1}^{\ell}(t^2-x_k^2)^m}{t^2-x_{\sigma(i)}^2}
			dt
			\ \ 
			\varepsilon_j \partial_{x_{\sigma(j)}}
			\right) \\
			&=\varepsilon_i x_{\sigma(i)}
			\sum_{j=1}^{\ell}
			\left(
			\varepsilon_j^2
			\int_0^{x_{\sigma(j)}} 
			\frac{\prod_{k=1}^{\ell}(t^2-x_k^2)^m}{t^2-x_{\sigma(i)}^2}
			dt
			\ 
			\partial_{x_{\sigma(j)}}
			\right) \\
			&=\varepsilon_i\theta_{\sigma(i)}^{(m)}
			=\Xi_m(\varepsilon_i e_{\sigma(i)})
			=\Xi_m(we_i).
		\end{align*}
		Thus $\Xi_m$ is $W$-equivariant.  
		By Theorem \ref{thm:FWY}, 
		$\theta_1^{(m)},\dots,\theta_\ell^{(m)}$ form a basis of 
		$\D(B_\ell,2m)_d$, 
		so $\Xi_m$ is an isomorphism.
	\end{proof}
	
	Let $\alpha_1^*,\dots,\alpha_\ell^*\in V$ be dual to the simple roots
	under the natural pairing, so that
	$\alpha_j(\alpha_r^*)=\delta_{jr}$ (Kronecker’s delta).
	We have
	$
	\alpha_r^*=e_1+\cdots+e_r
	\ (1\leq r\leq\ell).
	$
	Define the linear isomorphism
	\begin{equation}
		\label{eq:Linear-iso}
		\mathcal{L}_m
		=\res_d^{-1}\circ\Xi_m:
		V\longrightarrow \D_0(\cShi(B_\ell,m))_d
	\end{equation}
	and set
	$
	\varphi_0^+=0,
	\varphi_r^+=\mathcal{L}_m(\alpha_r^*)
	\ (1\leq r\leq\ell).
	$
	The derivations $\varphi_r^+$ are called a simple-root basis plus \cite{AT}.
	Theorem 3.5 in \cite{AT} gives
	\begin{equation}
		\label{eq:sr-varphi}
		s_r(\varphi_j^+)=\varphi_j^+
		\qquad(j\ne r).
	\end{equation}
	Proposition~4.2 and Theorem~4.3 in \cite{AT} show that the
	$m$-Euler derivation
	\begin{equation}
		\label{eq:m-Euler}
		\eta^{(m)}
		=\sum_{j=1}^{\ell}(\alpha_j+mz)\varphi_j^+
	\end{equation}
	is $W$-invariant.  
	We now derive the identity used in the recursion.
	Set
	$
	c_{rj}
	=\frac{2(\alpha_r,\alpha_j)}{(\alpha_r,\alpha_r)}.
	$
	By \eqref{eq:simple-reflection},
	$s_r(\alpha_j)=\alpha_j-c_{rj}\alpha_r$.  
	Apply $s_r$ to
	\eqref{eq:m-Euler}.  
	Since $s_r(\eta^{(m)})=\eta^{(m)}$,
	$s_r(\alpha_r)=-\alpha_r$, $s_r(z)=z$, 
	and
	\eqref{eq:sr-varphi} holds,
	we have
	\begin{align*}
		\eta^{(m)}
		&=(mz-\alpha_r)s_r(\varphi_r^+)
		+\sum_{j\ne r}
		\bigl(\alpha_j-c_{rj}\alpha_r+mz\bigr)\varphi_j^+.
	\end{align*}
	Comparing this with the original expression
	$
	\eta^{(m)}
	=(mz+\alpha_r)\varphi_r^+
	+\sum_{j\ne r}(\alpha_j+mz)\varphi_j^+
	$
	and cancelling the common terms gives
	\begin{equation}
		\label{eq:sr-m-Euler}
		(mz-\alpha_r)s_r(\varphi_r^+)
		=(mz+\alpha_r)\varphi_r^+
		+\alpha_r\sum_{j\ne r}
		\frac{2(\alpha_r,\alpha_j)}{(\alpha_r,\alpha_r)}
		\varphi_j^+.
	\end{equation}
	
	\section{Proof of the main theorem}
	\label{sec:proof}
	Let $\cCat(B_\ell,m)$ be the central arrangement with defining polynomial
	\begin{equation}
		\label{eq:Catalan}
		Q_{\cCat}(x,z)=z\prod_{\alpha\in\Phi^+(B_\ell)}C_m(\alpha(x),z),
	\end{equation}
	and let $\D_0(\cCat)$ denote the derivations that annihilate $z$.
	
	\begin{theorem}
		\cite{K}
		\label{thm:kawanoue}
		The vector field $\widetilde\kappa_{\ell,m}$ in
		\eqref{eq:k-basis} is homogeneous of degree $d+1$ and belongs
		to $\D_0(\cCat(B_\ell,m))$.
	\end{theorem}
	
	\begin{lemma}
		\label{lem:res-k}
		For the vector field $\widetilde{\kappa}_{\ell,m}$ defined in \eqref{eq:k-basis},
		we have
		$\res(\widetilde{\kappa}_{\ell,m})
		=c_{\ell,m}\eta_0^{(m)},$
		where
		$c_{\ell,m}=(-1)^{\ell m}2(m+1).$
		
	\end{lemma}
	
	\begin{proof}
		It is enough to compare one coefficient at a time.  
		Fix $i$, rename
		$x_i$ as $x$, and denote the other variables by
		$y_1,\dots,y_{\ell-1}$.
		Let $C= C_{\ell-1}$.
		Since
		$
		P_z(u,r)\big|_{z=0}=u^{2r+1},
		$
		we have
		$
		P_z(x,m+C)\big|_{z=0}=x^{2m+2C+1}.
		$
		Since $C_j=C_{j-1}+c_j$,
		\begin{equation*}
			\left.
			\frac
			{ P_z(y_j,m+C_{j-1}) }
			{ P_z(y_j,C_j) }
			\right|_{z=0}
			=y_j^{2(m+C_{j-1})+1-(2C_j+1)} 
			=y_j^{2(m-c_j)}.
		\end{equation*}
		Consequently, the coefficient of
		$
		x^{2m+2C+1}\prod_{j=1}^{\ell-1}y_j^{2(m-c_j)}
		$
		in $\res(\widetilde F_{\ell,m,0,i})$ is
		\begin{equation}
			\label{eq:L}
			L_C
			=\frac{(-1)^C \prod_{j=1}^{\ell-1}\binom{m}{c_j}}
			{\binom{m+C+\frac12}{m+1}}.
		\end{equation}
		
		Consider
		the polynomial
		$
		\int_0^x(t^2-x^2)^m
		\prod_{j=1}^{\ell-1}(t^2-y_j^2)^m dt
		$
		of
		$\eta_0^{(m)}$.
		For each $j$,
		$
		(t^2-y_j^2)^m
		=\sum_{c_j=0}^{m}
		(-1)^{m-c_j}\binom{m}{c_j}
		t^{2c_j}y_j^{2(m-c_j)}.
		$
		Thus the coefficient of
		$\prod_jy_j^{2(m-c_j)}$ in the product of the $y_j$-factors is
		\begin{equation}
			\label{eq:y-coefficient}
			(-1)^{(\ell-1)m-C}
			\prod_{j=1}^{\ell-1}\binom{m}{c_j}\,t^{2C}.
		\end{equation}
		
		Expand the remaining factor as
		$
		(t^2-x^2)^m
		=\sum_{r=0}^{m}
		(-1)^{m-r}\binom{m}{r}t^{2r}x^{2(m-r)}.
		$
		After multiplication by $t^{2C}$, 
		the $r$-th term contributes
		\begin{equation*}
			\int_0^x t^{2(C+r)}x^{2(m-r)} dt
			=\frac{x^{2(C+r)+1+2(m-r)}}{2(C+r)+1}
			=\frac{x^{2m+2C+1}}{2(C+r)+1}.
		\end{equation*}
		In particular, every $r$ contributes to the same power of $x$ appearing
		in 
		$
		x^{2m+2C+1}\prod_{j=1}^{\ell-1}y_j^{2(m-c_j)}
		$.  
		Note that
		$\int_0^x t^{2C}(t^2-x^2)^m dt = M_C \cdot x^{2m+2C+1}$
		where $M_C= \int_0^1 t^{2C}(t^2-1)^m dt$.
		By the substitution $s=t^2$, 
		we have
		\begin{align}
			\label{eq:MC-beta}
			M_C
			=\int_0^1 t^{2C}(t^2-1)^m dt 
			=\frac{(-1)^m}{2}
			\int_0^1s^{C-\frac12}(1-s)^m ds 
			=\frac{(-1)^m}{2}
			B\left(C+\frac12,m+1\right),
		\end{align}
		where $B$ is the beta function.  
		Using
		\[
		B\left(C+\frac12,m+1\right)
		=\frac{\Gamma(C+\frac12)\Gamma(m+1)}
		{\Gamma(C+m+\frac32)}
		=\frac{1}
		{(m+1)\binom{m+C+\frac12}{m+1}},
		\]
		equation \eqref{eq:MC-beta} becomes
		\begin{equation}
			\label{eq:SC-expansion}
			M_C=
			\frac{(-1)^m}
			{2(m+1)\binom{m+C+\frac12}{m+1}}.
		\end{equation}
		
		Combining \eqref{eq:y-coefficient} and \eqref{eq:SC-expansion}, the coefficient
		of the monomial 
		$
		x^{2m+2C+1}\prod_{j=1}^{\ell-1}y_j^{2(m-c_j)}
		$ in 
		$
		\int_0^x(t^2-x^2)^m
		\prod_{j=1}^{\ell-1}(t^2-y_j^2)^m dt
		$
		is
		\begin{equation}
			\label{eq:C}
			N_C
			=\frac{(-1)^{\ell m-C}
				\prod_{j=1}^{\ell-1}\binom{m}{c_j}}
			{2(m+1)\binom{m+C+\frac12}{m+1}}.
		\end{equation}
		Comparing \eqref{eq:L} and \eqref{eq:C} yields
		$
		\frac{L_C}{N_C}
		=2(m+1)(-1)^{C-(\ell m-C)}
		=(-1)^{\ell m}2(m+1).
		$
		
		The exponents of the $y_j$ uniquely determine the tuple
		$(c_1,\dots,c_{\ell-1})$, and these tuples run over all monomials in both
		polynomials.  Hence, we have
		$
		\res(\widetilde F_{\ell,m,0,i})
		=(-1)^{\ell m}2(m+1)
		\int_0^{x_i}\prod_{k=1}^{\ell}(t^2-x_k^2)^m dt,
		$
		and
		$
		\res(\widetilde\kappa_{\ell,m})
		=(-1)^{\ell m}2(m+1)\eta_0^{(m)}.
		$
	\end{proof}
	
	Define $(\tau_v f) (x,z) = f(x-zv,z)$, for $f\in \widetilde{S}$, $v\in\C^\ell$.
	We can apply $\tau_v$ coefficientwise to vector field $\eta=\sum_i f_i \partial_{x_i}$.
	Here $\eta (z) = 0$.
	We have $\rho_v^z \eta =\displaystyle\frac{\eta-\tau_v \eta }{z}.$
	
	\begin{lemma}
		\label{lem:transfer}
		Let $v\in\C^\ell$ and suppose that
		$\alpha(v)\in\{0,1\}$ for every $\alpha\in\Phi^+(B_\ell)$.  If
		$\eta\in \D_0(\cCat(B_\ell,m))$ is homogeneous of degree $q$, then
		$
		\rho_v^z\eta\in \D_0(\cShi(B_\ell,m))_{q-1}.
		$
	\end{lemma}
	
	\begin{proof}
		For every $\alpha\in\Phi^+(B_\ell)$ and $-m\leq k\leq m$,
		$
		\eta(\alpha-kz)=\eta(\alpha),
		$
		because $\eta(z)=0$.  Membership in $\D_0(\cCat)$ implies
		that every factor $\alpha-kz$ divides $\eta(\alpha)$.  These factors are
		pairwise relatively prime, so
		$C_m(\alpha,z)\mid\eta(\alpha)$.
		
		If $\alpha=\sum_{i=1}^{\ell}c_i x_i$ is a linear form,
		$\eta=\sum_i f_i \partial_{x_i}$, 
		then the constants
		$c_i=\partial_{x_i}\alpha$ are fixed by $\tau_v$.  
		Hence $\tau_v$ commutes with evaluation of $\eta$ on $\alpha$:
		\begin{equation}
			\label{eq:tau-evaluation}
			(\tau_v\eta)(\alpha)
			=\sum_{i=1}^{\ell}\tau_v(f_i)\partial_{x_i}(\alpha)
			=\sum_{i=1}^{\ell}\tau_v(f_i) c_i
			=\tau_v\left(\sum_{i=1}^{\ell}f_i c_i\right)
			=\tau_v\bigl(\eta(\alpha)\bigr).
		\end{equation}
		Set $\varepsilon=\alpha(v)$.
		The commutation identity shows
		\begin{equation}
			\label{eq:tau-alpha}
			(\tau_v\eta)(\alpha)=\tau_v\bigl(\eta(\alpha)\bigr)
			=\eta(\alpha)(x-zv,z).
		\end{equation}
		Moreover,
		$
		\alpha(x-zv)=\alpha(x)-z\alpha(v)=\alpha(x)-\varepsilon z.
		$
		Since $C_m(\alpha,z)\mid\eta(\alpha)$,
		we have
		$
		C_m(\alpha-\varepsilon z,z)
		\mid (\tau_v\eta)(\alpha).
		$
		If $\varepsilon=0$, 
		both $\eta(\alpha)$ and $(\tau_v\eta)(\alpha)$ are divisible by $C_m(\alpha,z)$, 
		and hence by
		$R_m(\alpha,z)$.  
		Suppose now that $\varepsilon=1$.  
		For an indeterminate $t$,
		the common factors 
		of $C_m(t,z)$ and $C_m(t-z,z)$
		are precisely those with indices
		$1-m,\dots,m$.  
		Hence, their greatest common divisor is
		$\prod_{s=1-m}^{m}(t-sz)
		=R_m(t,z).$
		Thus, in both cases $\varepsilon\in\{0,1\}$, 
		we have
		$
		R_m(\alpha,z)\mid
		\eta(\alpha)-(\tau_v\eta)(\alpha).
		$
		
		Since the polynomial $z$ is relatively prime to $R_m(\alpha,z)$,
		we have 
		$
		R_m(\alpha,z)\mid (\rho_v^z \eta )(\alpha).
		$
		This proves
		$
		\rho_v^z\eta
		\in \D_0(\cShi(B_\ell,m))_{q-1}.
		$
	\end{proof}
	
	\begin{remark}
		\label{re:v}
		If $v= \sum_i v_i e_i \in\C^\ell$ satisfies
		$\alpha(v)\in\{0,1\}$ for every
		$\alpha\in\Phi^+(B_\ell)$, 
		then $v=0$ or $v=e_1$.
		The coordinate roots give $v_i\in\{0,1\}$.  
		The roots $x_i+x_j$ imply
		that at most one coordinate equals $1$.  
		If $v_r=1$ for some $r>1$, 
		then
		$x_i-x_r$ has value $-1$ for every $i<r$, 
		a contradiction.  
		Hence the
		only nonzero possibility is $e_1$.
	\end{remark}
	
	Let
	$f\in \widetilde{S}$ 
	and set $\partial_v=\sum_i v_i\partial_{x_i}$.  
	Since $f$ is a
	polynomial, 
	its Taylor expansion in the $x$-direction $v$ is finite:
	\begin{equation}
		\label{eq:f-Taylor}
		f(x-zv,z)
		=\sum_{i\geq0}\frac{(-z)^i}{i!}
		(\partial_v^i f)(x,z).
	\end{equation}
	Subtracting \eqref{eq:f-Taylor} from $f(x,z)$ and dividing by $z$ gives
	\begin{equation}
		\label{eq:rho-Taylor}
		\rho_v^z f
		=\sum_{i\geq1}
		\frac{(-1)^{i+1}z^{i-1}}{i!}(\partial_v^i f)(x,z).
	\end{equation}
	After setting $z=0$, only the term $i=1$ remains.  
	Therefore
	\begin{equation}
		\label{eq:res-rho}
		\res(\rho_v^z f)
		=(\partial_v f)(x,0)
		=\partial_v(f(x,0))
		=\partial_v(\res f).
	\end{equation} 
	Applying \eqref{eq:res-rho} to every coefficient of
	$\eta=\sum_i f_i \partial_{x_i}$ yields
	\begin{equation}
		\label{eq:res-rho-commutation}
		\res(\rho_v^z\eta)
		=\sum_i\res(\rho_v^zf_i)\partial_{x_i}
		=\sum_i\partial_v(\res f_i)\partial_{x_i}
		=\bigtriangledown_{\partial_v}(\res\ \eta).
	\end{equation}
	
	\begin{theorem}
		\label{thm:res-Phi}
		Let
		$
		\lambda_{\ell,m}=\frac{(-1)^{\ell m+1}}{4m(m+1)},
		\Phi_1^{(m)}=\lambda_{\ell,m}\rho_{e_1}^z
		\widetilde\kappa_{\ell,m}.
		$
		The derivation $\Phi_1^{(m)}$ satisfies
		$
		\Phi_1^{(m)}\in \D_0(\cShi(B_\ell,m))_d,$
		and
		$
		\res(\Phi_1^{(m)})=\theta_1^{(m)}.
		$
		Consequently,
		$
		\Phi_1^{(m)}=\res_d^{-1}(\theta_1^{(m)}).
		$
	\end{theorem}
	
	\begin{proof}
		By Theorem \ref{thm:kawanoue}, Lemma \ref{lem:transfer}, and Remark \ref{re:v},
		we have $
		\Phi_1^{(m)}\in \D_0(\cShi(B_\ell,m))_d$.
		Lemma
		\ref{lem:eta0} and \ref{lem:res-k} together with equation \eqref{eq:res-rho-commutation} give
		\begin{align}
			\label{eq:res-Phi}
			\res (\Phi_1^{(m)} )
			&=\lambda_{\ell,m} 
			\res ( \rho_{e_1}^z \widetilde\kappa_{\ell,m} ) \notag\\
			&=\lambda_{\ell,m} \bigtriangledown_{\partial_{x_1}}
			\res(\widetilde\kappa_{\ell,m})\notag\\
			&=\lambda_{\ell,m} c_{\ell,m} \bigtriangledown_{\partial_{x_1}}
			\eta_0^{(m)}\notag\\
			&=-2m \lambda_{\ell,m} c_{\ell,m} \theta_1^{(m)}\notag\\
			&=\theta_1^{(m)}.
		\end{align}
	\end{proof}
	
	\begin{theorem}
		\label{thm:derivation-recursion}
		Let $\Phi_r^{(m)}$ be defined by the preceding initial conditions and recursion  \eqref{eq:Phi-recursion},
		and let the derivations $\varphi_r^+$ be a simple-root basis plus
		introduced in Section \ref{subsec:Weyl-group}.
		For $0\leq r\leq\ell$,
		we have 
		$
		\Phi_r^{(m)}=\varphi_r^+.
		$
		Consequently, 
		every division by $\alpha_r$ in
		\eqref{eq:Phi-recursion} is coefficientwise exact, 
		and
		$
		\Phi_r^{(m)}\in \D_0(\cShi(B_\ell,m))_d
		\ (1\leq r\leq\ell).
		$
		Moreover, if
		$
		\Theta_i^{(m)}=\Phi_i^{(m)}-\Phi_{i-1}^{(m)},
		$
		then
		$
		\Theta_i^{(m)}=\mathcal L_m(e_i),
		\res_d(\Theta_i^{(m)})=\theta_i^{(m)}
		\ (1\leq i\leq\ell).
		$
	\end{theorem}
	
	\begin{proof}
		For $1\leq r<\ell$, 
		the root
		$\alpha_r=x_r-x_{r+1}$ has square length $2$.  
		For $j \neq r$,
		direct
		calculation gives
		\begin{equation}
			\label{eq:B-inner-product}
			\frac{2(\alpha_r,\alpha_j)}{(\alpha_r,\alpha_r)}
			=
			\begin{cases}
				-1,&j=r-1\text{ and }r>1,\\
				-1,&j=r+1,\\
				0,&\text{otherwise}.
			\end{cases}
		\end{equation}
		
		With the set $\varphi_0^+=0$, equation
		\eqref{eq:B-inner-product} turns
		\eqref{eq:sr-m-Euler} into
		\begin{equation}
			\label{eq:sr-m-Euler-exact}
			(mz-\alpha_r)s_r(\varphi_r^+)
			=(mz+\alpha_r)\varphi_r^+
			-\alpha_r\bigl(\varphi_{r-1}^++\varphi_{r+1}^+\bigr).
		\end{equation}
		Equivalently,
		\begin{equation}
			\label{eq:exact-divisibility}
			(mz+\alpha_r)\varphi_r^+
			-(mz-\alpha_r)s_r(\varphi_r^+)
			=\alpha_r\bigl(\varphi_{r-1}^++\varphi_{r+1}^+\bigr).
		\end{equation}
		It proves, coefficient by coefficient, that the
		left-hand side is divisible by $\alpha_r$.  Solving
		\eqref{eq:exact-divisibility} for $\varphi_{r+1}^+$ gives
		\begin{equation}
			\label{eq:AT-recursion-solved}
			\varphi_{r+1}^+
			=\frac{
				(\alpha_r+mz)\varphi_r^+
				-(mz-\alpha_r)s_r(\varphi_r^+)
			}{\alpha_r}
			-\varphi_{r-1}^+.
		\end{equation}
		
		It remains to compare initial conditions.  
		By
		Theorem \ref{thm:res-Phi} and the definitions of $\mathcal{L}_m, \varphi_r^+$,
		$
		\Phi_1^{(m)}=\mathcal{L}_m(e_1)=\varphi_1^+,
		$
		while $\Phi_0^{(m)}=\varphi_0^+=0$. 
		Suppose for some
		$1\leq r<\ell$ that
		$
		\Phi_{r-1}^{(m)}=\varphi_{r-1}^+,
		\Phi_r^{(m)}=\varphi_r^+.
		$
		Substitution into \eqref{eq:Phi-recursion}, 
		followed by
		\eqref{eq:AT-recursion-solved}, 
		gives
		$\Phi_{r+1}^{(m)}=\varphi_{r+1}^+$.  
		Induction proves
		$\Phi_{r}^{(m)}=\varphi_{r}^+$ for $0\leq r\leq\ell$.  
		Since each
		$\varphi_r^+$ belongs to $\D_0(\cShi(B_\ell,m))_d$, 
		polynomiality and homogeneity of degree $d$ follow at the same time.
		
		Set $\alpha_0^*=0$.  
		Note that
		$e_i=\alpha_i^*-\alpha_{i-1}^*$.  
		By linearity of $\mathcal L_m$,
		we have
		$
		\Theta_i^{(m)}
		=\Phi_i^{(m)}-\Phi_{i-1}^{(m)}
		=\mathcal L_m(\alpha_i^*-\alpha_{i-1}^*)
		=\mathcal L_m(e_i).
		$
		Applying $\res_d$ and using
		$\res_d\circ\mathcal L_m=\Xi_m$ yields
		$
		\res_d(\Theta_i^{(m)})
		=\Xi_m(e_i)=\theta_i^{(m)}.
		$
	\end{proof}
	
	By the above theorem, we give the explicit construction of 
	simple-root basis plus $\varphi_r^+= \Phi_{r}^{(m)}$. 
	This partially answers the question posed by Abe and Terao.
	We can now prove main theorem \ref{thm:main}.
	
	\begin{proof}[Proof of Theorem \ref{thm:main}]
		By Theorem \ref{thm:FWY}, \ref{thm:restriction}, and \ref{thm:derivation-recursion},
		the proof is complete.
	\end{proof}
	
	\section{Examples}
	
	Consider the case $B_2$, $m=1$.
	Set $x=x_1$ and $y=x_2$.  
	Recall that
	$
	P_z(a,r)=\prod_{k=-r}^{r}(a-kz)
	=a\prod_{k=1}^{r}(a^2-k^2z^2).
	$
	In particular,
	\begin{equation*}
		\begin{aligned}
			P_z(a,0)&=a,\\
			P_z(a,1)&=a(a^2-z^2),\\
			P_z(a,2)&=a(a^2-z^2)(a^2-4z^2).
		\end{aligned}
	\end{equation*}
	In equation \eqref{eq:k-h}, there is only one summation index
	$c=c_1\in\{0,1\}$, with $C_0=0$ and $C_1=c$.  Therefore
	\begin{equation*}
		\widetilde h_{1,1,0}(x;y;z)
		=\sum_{c=0}^{1}
		\frac{(-1)^cP_z(x,1+c)}
		{\binom{1+c+\frac12}{2}}
		\binom{1}{c}
		\frac{P_z(y,1)}{P_z(y,c)}.
	\end{equation*}
	For $c=0$, we have
	$
	\binom{\frac32}{2}
	=\frac{(3/2)(1/2)}{2}=\frac38,
	\frac{P_z(y,1)}{P_z(y,0)}=y^2-z^2.
	$
	Thus the $c=0$ summand is
	$
	T_0
	=\frac{P_z(x,1)}{3/8}
	\frac{P_z(y,1)}{P_z(y,0)}\\
	=\frac83x(x^2-z^2)(y^2-z^2).
	$
	For $c=1$, we have
	$
	\binom{\frac52}{2}
	=\frac{(5/2)(3/2)}{2}=\frac{15}{8},
	\frac{P_z(y,1)}{P_z(y,1)}=1.
	$
	So the second summand is
	$
	T_1
	=-\frac{P_z(x,2)}{15/8}
	=-\frac8{15}x(x^2-z^2)(x^2-4z^2).
	$
	Combining $T_0$ and $T_1$, we obtain
	\begin{equation*}
		\widetilde h_{1,1,0}(x;y;z)
		=T_0+T_1
		=\frac8{15}x(x^2-z^2)
		\left[5(y^2-z^2)-(x^2-4z^2)\right]
		=\frac8{15}x(x^2-z^2)(5y^2-x^2-z^2).
	\end{equation*}
	By definition,
	\begin{align*}
		\widetilde\kappa_{2,1}
		&=\widetilde h_{1,1,0}(x;y;z)\partial_x
		+\widetilde h_{1,1,0}(y;x;z)\partial_y \\
		&=\frac8{15}
		\Bigl[
		(-x^5+5x^3y^2-5xy^2z^2+xz^4)\partial_x
		+(-y^5+5x^2y^3-5x^2yz^2+yz^4)\partial_y
		\Bigr].
	\end{align*}
	Restriction to the hyperplane $z=0$ gives
	\begin{equation*}
		\res(\widetilde\kappa_{2,1})
		=\frac8{15}\left[
		x^3(5y^2-x^2)\partial_x
		+y^3(5x^2-y^2)\partial_y
		\right].
	\end{equation*}
	Indeed,
	\begin{equation*}
		\int_0^x(t^2-x^2)(t^2-y^2)dt
		=\frac2{15}x^3(5y^2-x^2).
	\end{equation*}
	We have
	$
	\res(\widetilde\kappa_{2,1}) =4\eta_0^{(1)},
	$
	in agreement with
	$c_{2,1}=(-1)^{2}\,2(1+1)=4$.
	Let $A(x,y,z)=x(x^2-z^2)(5y^2-x^2-z^2)$, $B(x,y,z)=y(y^2-z^2)(5x^2-y^2-z^2)$.
	We have
	$
	\widetilde\kappa_{2,1}=\frac8{15}(A(x,y,z)\partial_x+B(x,y,z)\partial_y).
	$
	For $\ell=2$ and $m=1$, 
	$\lambda_{2,1}=\frac{(-1)^{2\cdot1+1}}{4\cdot1\cdot(1+1)}=-\frac18.$
	Introduce the two vector fields
	\begin{align*}
		\Psi_1=\ 
		&
		x(x-z)(-x^2+xz+3y^2-z^2)\partial_x
		+y(y-z)(2x-z)(y+z)\partial_y, \\
		\Psi_2=\ 
		&
		x(x-z)(2x-z)(y-z)\partial_x
		+y(y-z)(3x^2-3xz+z^2-y^2)\partial_y.
	\end{align*}
	In fact, by Proposition \ref{prop:szc}, $\theta_E$, $\Psi_1$, $\Psi_2$ form a basis for $\D(\cShi(B_2,1))$.
	Direct calculations give
	\begin{align*}
		&\rho_{e_1}^z A(x,y,z)
		=\ 5x(x-z)(-x^2+xz+3y^2-z^2).\\
		&\rho_{e_1}^z B(x,y,z)
		\ =5y(y-z)(2x-z)(y+z).
	\end{align*}
	We obtain
	$
	\Phi_1^{(1)}
	=\lambda_{2,1}\rho_{e_1}^{z}
	\widetilde\kappa_{2,1}
	=-\frac18\cdot\frac8{15}\cdot5\Psi_1
	=-\frac13\Psi_1.
	$
	
	Let $s_1$ interchange $x$ with $y$ and $\partial_x$ with $\partial_y$.  
	Thus
	$
	s_1(\Psi_1)= 
	x(x-z)(2y-z)(x+z)\partial_x
	+y(y-z)(-y^2+yz+3x^2-z^2)\partial_y.
	$
	Note that $\alpha_1 =x-y$. 
	By equation \eqref{eq:Phi-recursion}, we obtain
	\begin{align*}
		\Phi_2^{(1)}
		&=\frac{(x-y+z)\Phi_1^{(1)}
			-(z-x+y)s_1(\Phi_1^{(1)})}{x-y},\\
		&=-\frac13
		\frac{(x-y+z)\Psi_1-(z-x+y)s_1(\Psi_1)}{x-y} \\
		&=-\frac13\Bigl[
		x(x-z)(-x^2+2xy+3y^2-xz-yz)\partial_x
		+y(y-z)(3x^2+2xy-y^2-xz-yz)\partial_y\Bigr],\\
		&=-\frac13(\Psi_1+\Psi_2).
	\end{align*}
	
	Since
	$
	\Theta_i^{(1)}=\Phi_i^{(1)}-\Phi_{i-1}^{(1)},
	$
	we have
	\begin{align*}
		&\Theta_1^{(1)}
		=\Phi_1^{(1)}=-\frac13\Psi_1,\\
		&\Theta_2^{(1)}
		=\Phi_2^{(1)}-\Phi_1^{(1)} 
		=-\frac13\Psi_2, \\
		&\res(\Theta_1^{(1)})
		=(\tfrac13x^4-x^2y^2)\partial_x
		-\tfrac23xy^3\partial_y 
		=\theta_1^{(1)} ,\\
		&\res(\Theta_2^{(1)})
		=-\tfrac23x^3y\partial_x
		+(\tfrac13y^4-x^2y^2)\partial_y
		=\theta_2^{(1)} .
	\end{align*}
	The derivations $\Theta_1^{(1)}, \Theta_2^{(1)}$ together with the Euler derivation $\theta_E$ form a basis 
	for $\D(\cShi(B_2,1))$.

\end{document}